\documentclass[a4paper,12pt]{amsart}
\usepackage{amsthm, amsfonts, amssymb, mathtools, color}
\usepackage[all]{xy}
\usepackage{fullpage}
\usepackage{bm}
\usepackage{url}
\usepackage{xcolor}
\usepackage{hyperref}
\hypersetup{colorlinks=true,linkcolor=blue!45!black,
  citecolor=blue!45!black,urlcolor=blue!45!black,
  pdftitle={Non-vanishing of the structure-sheaf cohomology of affine henselian schemes and henselian affinoids}}
\def\cA{\mathcal{A}}

\def\cF{\mathcal{F}}
\def\cO{\mathcal{O}}
\def\cU{\mathcal{U}}
\def\cX{\mathcal{X}}

\def\fm{\mathfrak{m}}
\def\fp{\mathfrak{p}}
\def\fq{\mathfrak{q}}

\def\la{\lambda}
\def\La{\Lambda}

\def\ka{\kappa}
\def\til#1{\widetilde{#1}}
\def\ovl#1{\overline{#1}}
\def\what#1{\widehat{#1}}

\DeclareMathOperator{\Spec}{Spec}

\DeclareMathOperator{\Sph}{Sph}
\DeclareMathOperator{\Frac}{Frac}

\DeclareMathOperator{\Aut}{Aut}

\DeclareMathOperator{\Tr}{Tr}
\DeclareMathOperator{\Jac}{Jac}
\DeclareMathOperator{\HH}{H}
\def\het#1{{}^h\kern-.12em{#1}}
\def\dl{\langle\kern-.22em\langle}
\def\dr{\rangle\kern-.22em\rangle}
\def\dal{\langle\kern-.22em\langle}
\def\dar{\rangle\kern-.22em\rangle}
\def\dbl{[\kern-.12em[}
\def\dbr{]\kern-.12em]}
\def\dpl{(\kern-.2em(}
\def\dpr{)\kern-.2em)}

\def\dast{\ast\kern-.03em\ast}
\def\tast{\ast\kern-.26em\ast\kern-.26em\ast}

\def\longto{\longrightarrow}
\def\longhookrightarrow{\lhook\joinrel\longto}

\def\injto{\hookrightarrow}
\def\injlongto{\longhookrightarrow}
\newcommand{\tagref}[1]{\href{https://stacks.math.columbia.edu/tag/#1}{Tag~#1}}
\theoremstyle{plain}
\newtheorem{thm}[subsection]{Theorem}
\newtheorem{prop}[subsection]{Proposition}
\newtheorem{lem}[subsection]{Lemma}
\newtheorem{cor}[subsection]{Corollary}

\theoremstyle{definition}
\newtheorem{dfn}[subsection]{Definition}

\newtheorem{exa}[subsection]{Example}

\theoremstyle{remark}
\newtheorem{rem}[subsection]{Remark}

\title[Non-vanishing of the structure-sheaf cohomology]
{Non-vanishing of the structure-sheaf cohomology of affine henselian schemes and henselian affinoids}
\author{Yoshinori Gongyo}
\address{Graduate School of Mathematical Sciences, The University of Tokyo, 3-8-1 Komaba, Meguro-ku, Tokyo, 153-8914, Japan}
\email{gongyo@ms.u-tokyo.ac.jp}

\author{Fumiharu Kato}
\address{Faculty of Social Informatics, ZEN University, 3-12-11 Shinjuku, Zushi-shi, Kanagawa, 249-0007, Japan}
\email{fumiharu\_kato@zen.ac.jp}

\author{Takayuki Uchiba}
\address{sugakubunka Co., Ltd., 7-4-4 Nishi-Shinjuku, Shinjuku-ku, Tokyo 160-0023, Japan}
\email{takayuki.uchiba@sugakubunka.com}

\date{\today}

\begin{document}

\begin{abstract}
Let $(A,I)$ be a henselian pair such that $A$ is a domain
of characteristic zero and $I\ne0$. We prove that any flat affine
henselian scheme of finite presentation over $(A,I)$ has
nonzero first structure-sheaf cohomology whenever one of its
fibers over a point in $V(I)$ has positive dimension.
This in particular gives a negative answer to the question by Devadas in the mixed-characteristic case. As an application, we also prove that the first structure-sheaf cohomology does not vanish for henselian affinoids of positive dimensions.
\end{abstract}

\maketitle


\section{Introduction}\label{sec:intro}

In \cite[Theorem~1.12]{greco1981quasi}, Greco and Strano asserted an analogue of Theorem B for quasi-coherent
sheaves on affine henselian schemes. Later, de Jong constructed a
characteristic-zero counterexample to this theorem \cite{deJong2018TheoremB}, which is further reproduced and developed
by Devadas in \cite[Proposition~3.1.15 and Example~3.1.16]{devadas2026henselian}.
In positive characteristic, by contrast, higher cohomology vanishes
for the henselian pullback of any quasi-coherent sheaf on an affine
scheme \cite[Theorem~3.2.6]{devadas2026henselian}. 
De Jong also proved non-vanishing on the henselian projective line over
a complete discrete valuation ring of characteristic zero \cite{dejong2019laurent,de2019coh}; see also
\cite[Example~4.4.1]{devadas2025gagahenselianschemes}. That result includes mixed characteristic cases. Devadas asks whether the Theorem B for affine henselian schemes might still
hold in the mixed characteristic case \cite[Question~1.1]{devadas2026henselian}.

Our main questions are as follows: Are affine Henselian schemes (in characteristic zero) such that their $\HH^1(\cO)$ do not vanish special, or are they ubiquitous? What about Henselian affinoids (in Henselian rigid geometry)?

The purpose of this paper is to clarify the answers to these questions and to Devadas’s question. Our main results are the following. 

\begin{thm}\label{thm:main1}
Let $(A,I)$ be a henselian pair with $A$ an integral domain of characteristic zero, and $I\ne0$. Let $(B,IB)$ be an $A$-flat henselian algebra of henselian finite presentation.
Suppose that there exists a point $\fp\in V(I)$ such that
\[
  \dim\bigl(B\otimes_A\ka(\fp)\bigr)>0.
\]
Then, for $X=\Sph(B,IB)$, we have
\[
  \HH^1(X,\cO_X)\ne0.
\]
\end{thm}

\begin{cor}\label{cor:general}
Let $(A,I)$ be a henselian pair, and let $(B,IB)$ be an $A$-flat henselian algebra of henselian finite presentation. Suppose that there are prime ideals $P\subseteq \fp$ of $A$ such that
\begin{equation*}\label{equ:hypotheses}
I\subseteq \fp,\quad I\nsubseteq P,\quad \operatorname{char}(A/P)=0,\quad\dim(B\otimes_A\ka(\fp))>0.
\end{equation*}
Then, for $X=\Sph(B,IB)$,
\[
\HH^1(X,\cO_X)\neq0.
\]
\end{cor}

Note that the first theorem gives a negative answer to \cite[Question~1.1]{devadas2026henselian}.

As for the proof, proving the first theorem is the essential part.
Let us outline the proof of Theorem \ref{thm:main1}.

\begin{dfn}\label{dfn:class}
In the situation as in Corollary \ref{cor:general}, each pair $(a,b)\in I\times B$ defines the \v{C}ech cohomology class in $\check{\HH}^1(\cU,\cO_X)$, where $\cU$ is the open covering $\cU=\{D(b),D(b-1)\}$ of $X=\Sph B$, induced from the section $f_{a,b}$ defined by
\begin{equation}\label{eq:intro-class}
f_{a,b}=1+2w,\quad
w^2+w=\frac{a}{b(b-1)},\quad w\equiv0\pmod{IB}.
\end{equation}
We denote by $\sigma_{a,b}\in\HH^1(X,\cO_X)$ the associated cohomology class in $\HH^1(X,\cO_X)$. 
\end{dfn}

Note that the section $w$ (and hence $f_{a,b}$ as well) is uniquely determined by the pair $(a,b)$ due to Hensel's lemma, and thus we have the following compatibility of $\sigma_{a,b}$: If 
$$
\xymatrix{(B,IB)\ar[r]&(B',IB')\\ (A,I)\ar[u]\ar[r]&(A',IA')\ar[u]}
$$
is a commutative diagram of henselian pairs, and $(a,b)$ is mapped to $(a',b')\in A'\times B'$, then the induced map 
$$
\HH^1(X,\cO_X)\longto\HH^1(X',\cO_{X'})
$$
(where $X'=\Sph(B',IB')$) of cohomologies maps $\sigma_{a,b}$ to $\sigma_{a',b'}$.

Now the proof of Theorem \ref{thm:main1} boils down to showing that there exists a pair $(a,b)\in I\times B$ such that $\sigma_{a,b}\neq 0$ in $\HH^1(X,\cO_X)$. 
The proof consists of the following three steps.

\subsection*{Step 1 (\S\ref{sec:line})}
We first consider the affine line case over a complete DVR $A=R$ with $I=\pi R$, where $\pi$ is a uniformizer, i.e., $B=(R[t])^h$ is the $\pi$-adic henselization of the polynomial ring over $R$. In this case, the core of the argument lies in showing that, for any non-zero $a\in I$, the cohomology class $\sigma_{a,t}$ is non-zero. This part of the proof is one of the most technical parts of our proof.

\subsection*{Step 2 (\S\ref{sec:generalNoetherian})}
Next, we consider the case where $A$ is a Noetherian $I$-adically henselian ring. In this case, we first choose a map $A\to R$ to a DVR in the fractional field of $A$ that dominates $\fp$, and consider the henselization $B_{\what{R}}=(B\otimes_A\what{R})^h$ of the base change to the completion of $R$. The key concept is what we call the {\em normalization tuple}, which is a tuple of elements in $B$ that gives Noether normalization of $B\otimes_A\ka(\fp)$. We will show that the map $\what{R}\{t\}\to B_{\what{R}}$ determined by the image of one of the members of the tuple, say $b$, induces the injective $\HH^1(\Sph\what{R}\{t\},\cO)\to\HH^1(\Sph B_{\what{R}},\cO)$. Since the non-zero image of $\sigma_{a,t}$ (as in Step 1) in $\HH^1(\Sph B_{\what{R}},\cO)$ coincides with the image of $\sigma_{a,b}$ of $\HH^1(\Sph B,\cO)$, we have $\sigma_{a,b}\neq 0$.

\subsection*{Step 3 (\S\ref{sec:general})}
We employ the usual technique to eliminate the Noetherian hypothesis. 
We take a filtered direct system $\{(A_{\la},I_{\la})\}_{\la\in\La}$ of henselian pairs with each $A_{\la}$ a Noetherian integral domain of characteristic zero that converges to $(A,I)$.
After restricting to a cofinal directed subset if necessary, this extends to a filtered direct system $\{(A_{\la},I_{\la})\to (B_{\la},I_{\la}B_{\la})\}_{\la\in\La}$ of flat and henselian finitely presented henselian algebras with the properties that satisfy the premises of the theorem.
The core part of the argument lies in the construction of the compatible system of normalization tuples of $B_{\la}$. As in Step 2, this gives rise to the compatible system of non-zero cohomology classes $\{\sigma_{a,b_{\la}}\in\HH^1(\Sph B_{\la},\cO)\}_{\la\in\La}$, defining a non-zero element in $\HH^1(\Sph B,\cO)$, as desired.

\medskip
Notice that our argument for the proof is independent from de Jong's example. 

After completing the proofs of Theorem \ref{thm:main1} and Corollary \ref{cor:general}, \S\ref{sec:exa} discusses interesting counterexamples that arise when we remove the assumptions of the theorem one by one.

In the course of proving the main theorem, we observe that the non-zero cohomology classes concentrate outside the ideal of definition (e.g., Corollary \ref{cor:line-torsion}). This leads us expect that the non-vanishing property will hold for associated Henselian affinoids as well. In fact, in the final section \S\ref{sec:affinoids}, we will prove the following theorem regarding this line (see \cite{kato2017henselianrigid} for basic notation in henselian rigid geometry):

\begin{thm}[Non-vanishing for positive-dimensional henselian affinoids]\label{thm:general-henselian-affinoid-nonvanishing}
Let $V$ be an $a$-adically henselian height-one valuation ring of characteristic zero, and $K=\Frac(V)$.
Let $\cA$ be a henselian affinoid $K$-algebra with $\dim\cA>0$, and let $\cX$ be its associated affinoid. Then
\[
  \HH^1(\langle\cX\rangle,\cO_{\cX})\neq 0.
\]
\end{thm}

\subsection*{Convention}
\begin{itemize}
\item All rings are assumed to be commutative with $1$, and all ring homomorphisms are assumed to send $1$ to $1$.
\item For a local ring $A$, $\fm_A$ denotes the maximal ideal of $A$.
\item For a ring $A$ and a prime ideal $\fp$, the residue field at $\fp$ is denoted by $\ka(\fp)$.
\item The sheaf $\cO$ in the notation of cohomologies $\HH^q(X,\cO)$ means the structure sheaf of the first entry $X$.
\item By a principal open subset $D(f)$ of $X=\Sph(A,I)$ by $f\in A$, we mean the usual scheme-theoretic $D(f)$ in $\Spec A$ restricted on the underlying topological space $V(I)$ of $X$.
\end{itemize}

\section{henselian pairs and henselian affine schemes}\label{sec:henselianAffineSchemes}
\subsection*{2A. henselian pairs}\label{sub:henselianPairs}

A {\em pair} always refers to a couple $(A,I)$ of a commutative ring $A$ and an ideal $I$ of it. An {\em ideal of definition} of a pair $(A,I)$ is an ideal $J\subset A$ such that the $J$-adic topology on $A$ coincides with the $I$-adic one, or equivalently, $I^m\subset J^n\subset I$ for some $n,m > 0 $. A {\em morphism} of pairs $f\colon (A,I)\to (B,J)$ is a ring homomorphism $f\colon A\to B$ such that $I^nB\subset J$ for some $n > 0$, i.e., $f$ is continuous with respect to the $I$-adic topology on $A$ and the $J$-adic topology on $B$. A morphism $f\colon (A,I)\to (B,J)$ is said to be {\em adic} if $IB$ is an ideal of definition of $(B,J)$. 

For the notion of {\em Zariskian} and {\em henselian} pairs, we refer to, e.g., \cite[Chap.~0, \S 7.3]{fujiwara_kato_2018}. Note that a henselian pair $(A,I)$ is Zariskian, i.e., $I$ is contained in the Jacobson radical $\Jac(A)$ of $A$. Note also that a complete pair is Henselian.

The following lemma is well-known; cf.\ \cite[Chap.~XI, \S 2, Prop.~2]{Raynaud1970}.

\begin{lem}\label{lem:henselianProperties}
{\rm (1)} If $(A,I)$ is a henselian pair, and $B$ is an integral $A$-algebra, then $(B, IB)$ is henselian.

{\rm (2)} If $\{(A_{\la},I_{\la})\}_{\la\in\La}$ is a filtered inductive system of henselian pairs, then $(A,I)=\varinjlim_{\la\in\La}(A_{\la},I_{\la})$ is henselian.
\end{lem}

There is a {\em henselization} for arbitrary pairs, characterized by a suitable universal mapping property; see \cite[Chap.~XI, \S 2, D\'ef.~4 \& Thm.~2]{Raynaud1970} or \cite[\S 2.8]{kurke1975henselsche}.
We denote by $A^h_I$, or $A^h$, the $I$-adic henselization of $A$, which then forms the henselian pair $(A^h,IA^h)$.

Let $(A,I)$ be henselian pair. We denote by
$$
A\{X_1,\ldots,X_n\}
$$
the $I$-adic henselization of the polynomial ring $A[X_1,\ldots,X_n]$.

\begin{dfn}\label{dfn-henselianfinitetype}{\rm 
We say that an $A$-algebra $B$ is {\em of henselian finite type} (resp.\ {\em of henselian finite presentation} ({\em hfp} for short)) if there exists a surjective $A$-algebra homomorphism
$$
A\{X_1,\ldots,X_n\}\longrightarrow B
$$
for some $n\geq 0$ $($resp.\ with finitely generated kernel$)$.}
\end{dfn}

Note that, due to Lemma \ref{lem:henselianProperties} (1), henselian finite type $A$-algebras are $I$-adically henselian.
We say that a morphism $(A,I)\to (B,J)$ between henselian pairs is of henselian finite type (resp.\ hfp) if it is adic and $B$ is of henselian finite type (resp.\ hfp) over $A$.

\begin{prop}\label{prop:background1}
Let $(A,I)$ be a henselian pair.
\begin{itemize}
\item[{\rm (1)}] If $B$ is hfp over $A$ and $B/IB$ is finite over $A/I$, then $B$ is finite over $A$.
\item[{\rm (2)}] The property hfp is stable under henselian base change. Any morphism over $(A,I)$ between two hfp henselian $A$-algebras is itself hfp.
\end{itemize}
\end{prop}

\begin{proof}
(1) See \cite[5.1.1]{devadas2026henselian}.

(2) See \cite[4.1.5 \& 4.1.7]{devadas2026henselian}.
\end{proof}

\subsection*{2B. henselian affine schemes}\label{sub:henselianSchemes}

For a henselian pair $(A,I)$, the {\em affine henselian scheme} $\Sph(A,I)$ has underlying topological space $\Spec(A/I)$ and structure sheaf given on principal open subsets by
\[
  \Gamma(D(g)\cap V(I),\cO)=(A_g)^h_{IA_g}.
\]
The sheaf property of the structure sheaf $\cO$ is well-known (cf.\ \cite[Prop.~2.1.2]{devadas2026henselian}).
More generally, one has the notion of the sheaf $\til{M}$ on $\Sph(A,I)$ associated to an $A$-module $M$ with sections 
\[
  \Gamma(D(g)\cap V(I),\til{M})=M\otimes_A(A_g)^h_{IA_g}
\]
on principal open subsets.

\begin{prop}\label{prop:background2}
Let $(A,I)$ be a henselian pair, and let $X=\Sph(A,I)$.
\begin{itemize}
\item[{\rm (1)}] For an $A$-module $M$, $\til{M}$ is a quasi-coherent sheaf on $X$, and we have $\Gamma(X,\til{M})=M$. Moreover, the functor $M\mapsto\til{M}$ from the category of $A$-modules to the category of $\cO_X$-modules is fully faithful and exact.
\item[{\rm (2)}] If $A\to B$ is finite and $\varphi\colon Y=\Sph(B,IB)\to\Sph(A,I)$ is the induced morphism, then $\varphi_{\ast}\cO_Y=\til{B}$ as $\cO_X$-modules.
\end{itemize}
\end{prop}

\begin{proof}
(1) The property $\Gamma(X,\til{M})=M$ follows from the sheaf property of $\til{M}$. From this, similarly to the scheme case, one can show that $M\mapsto\til{M}$ is left adjoint to $\Gamma_X\colon \cF\mapsto\Gamma(X,\cF)$. Hence $M\mapsto\til{M}$ preserves arbitrary direct sums, from which it follows that $\til{M}$ is quasi-coherent. All the rest is straightforward. 

(2) Since $\varphi^{-1}(D(f)\cap V(I))=D(g)\cap V(IB)$ for any $f\in A$ with the image $g$ in $B$, it suffices to show that $(B_g)^h_{IB_g}\cong B\otimes_A(A_f)^h_{IA_f}$, which follows easily from $B_g\cong B\otimes_AA_f$ and Lemma \ref{lem:henselianProperties} (1) (see also \cite[\tagref{0DYE}]{stacks-project}).
\end{proof}

\subsection*{2C. Some facts on the cohomology groups}\label{sub:torsion}

\begin{lem}\label{lem:torsion}
Let $(A,I)$ be a henselian pair, and $a\in A$ be a non-zero-divisor. Then multiplication by $a$ is injective on the cohomology $\HH^1(\Sph(A,I),\cO)$. In particular, if $A$ is an integral domain, this cohomology group is a torsion-free $A$-module.
\end{lem}

\begin{proof}
Consider the exact sequence
$$
0\longto A\stackrel{a}{\longto}A\longto A/aA\longto 0
$$
which induces the exact sequence
$$
0\longto\cO\stackrel{a}{\longto}\cO\longto\cO/a\cO\longto 0
$$
by Proposition \ref{prop:background2} (1).
The assertion follows from the cohomology long exact sequence.
\end{proof}

\begin{prop}\label{prop:dvr-cohomology}
Let $R$ be a ring, and $\pi\in R$.
Let $A$ be a $\pi$-torsion free $R$-algebra, and suppose $(A,\pi A)$ is henselian. Then, for $Y=\Sph(A,\pi A)$, the multiplication by $\pi$ is an isomorphism on $\HH^i(Y,\cO_Y)$ for every $i>0$. 
\end{prop}

\begin{proof}
Since both localization and henselization are flat and $A$ is $\pi$-torsion free, we have an exact sequence
\[
  0\longto\cO_Y\stackrel{\pi}{\longto}\cO_Y\longto\cO_{Y_0}\longto 0,
\]
where $Y_0$ is the usual scheme $\Spec A/\pi A$. 
Since the cohomology long exact sequence in degree zero gives the exact sequence
$$
0\longto A\stackrel{\pi}{\longto}A\longto A/\pi A\longto 0, 
$$
and since we have $\HH^i(Y_0,\cO_{Y_0})=0$ for all $i>0$, we have the assertion.
\end{proof}

\begin{rem}\label{rem:not-finite}
Note that, if $R$ is a local ring with $\fm_R=\pi R$ and one of the groups $\HH^i(Y,\cO_Y)$ for $i>0$ is nonzero, then it is not finitely generated over $R$ (due to Nakayama).
\end{rem}

\section{An explicit class on the henselian affine line}\label{sec:line}

\begin{lem}\label{lem:involutions}
Let $F$ be a field of characteristic zero. Suppose that $u,v$ are algebraic over $F$, that $f=u+v\neq 0$, and that $f^2\in F$. Then $f\in F(u)$ or $f\in F(v)$.
\end{lem}

\begin{proof}
Suppose that neither containment holds. The finite extension
\[
  L=F(u,v)=F(u,f)=F(v,f)
\]
is quadratic over each of $F(u)$ and $F(v)$. Let $\sigma$ be the involution fixing $v$ and negating $f$, and let $\tau$ be the involution fixing $u$ and negating $f$. Both are $F$-automorphisms of
$L$. From $f=u+v$ we calculate
\[
  \sigma(u)=u-2f,\quad
  (\sigma\tau)(f)=f,\quad
  (\sigma\tau)(u)=u-2f.
\]
Consequently $(\sigma\tau)^n(u)=u-2nf$ for every $n\geq 0$, which are distinct in characteristic zero. This contradicts the finiteness of $\Aut_F(L)$.
\end{proof}

\begin{thm}\label{thm:line}
Let $R$ be a henselian DVR of characteristic zero with uniformizer $\pi$. Consider the henselian affine line
\[
  L_R=\Sph\bigl((R[x])^h_{(\pi)},\pi(R[x])^h_{(\pi)}\bigr)
\]
over $(R,\pi)$. Then, for every $0\ne a\in\pi R$, the cohomology class $\sigma_{a,x}$ $($Definition {\rm \ref{dfn:class}}$)$ is non-zero in $\HH^1(L_R,\cO)$.
\end{thm}

The proof of the theorem uses the following lemmas.

\begin{lem}\label{lem:Cech}
For any ringed space $T$ with an open covering $T=U\cup V$ consisting of two open subsets, the map
\begin{equation}\label{eq:cech-injection}
 \frac{\Gamma(U\cap V,\cO_T)}
 {\operatorname{im}\bigl(\Gamma(U,\cO_T)\oplus\Gamma(V,\cO_T)\bigr)}
 \lhook\joinrel\longrightarrow \HH^1(T,\cO_T),
\end{equation}
where the map from the direct sum is the difference of the restrictions, is injective.
\end{lem}

\begin{proof}
This follows from the Mayer--Vietoris sequence for $\cU=\{U,V\}$ \cite[\tagref{01EB}]{stacks-project}.
\end{proof}

\begin{lem}\label{lem:gluingOfDomainness}
Let $A$ be a ring, and $a\in A$. Suppose $A$ is $a$-adically separated, i.e., $\bigcap_{n>0}a^nA=\{0\}$.

{\rm (1)} Let $f\colon A\to B$ be a ring homomorphism, and suppose that $B$ is $a$-torsion free and that the induced map $\ovl{f}\colon A/aA\to B/aB$ is injective. Then $f$ is injective.

{\rm (2)} If $A$ is $a$-torsion free and $A/aA$ is an integral domain, $A$ is an integral domain.
\end{lem}

\begin{proof}
(1) Suppose $f(b)=0$ for $b\in A$. Then $b\in aA$, i.e., $b=ab_1$ for some $b_1\in A$.
We have $f(ab_1)=af(b_1)=0$ and since $B$ is $a$-torsion free, we have $f(b_1)=0$.
Then, similarly, we have $b_1=ab_2$ for some $b_2\in A$, i.e., $b\in a^2A$.
In this way, one can inductively show that $b\in a^nA$ for any $n>0$, which implies $b=0$ since $A$ is $a$-adically separated.

(2) Take $b,c\in A$ such that $bc=0$ and $c\neq 0$. Since $A$ is $a$-adically separated, we have $c=c_1a^n$, $c_1\not\in aA$. Since $A$ is $a$-torsion free, $bc_1a^n=0$ implies $bc_1=0$. Since $A/aA$ is an integral domain, $b$ belongs to $aA$, i.e., $b=ab_1$ for some $b_1\in A$.
We have $a(b_1c_1)=0$, and since $A$ is $a$-torsion free, we have $b_1c_1=0$.
Then, similarly, we have $b_1=ab_2$ for some $b_2\in A$, i.e., $b\in a^2A$.
The rest of the proof is similarly to that of (1).
\end{proof}

\begin{rem}\label{rem:separated}
Note that if $A$ is Noetherian and $a\in \Jac(A)$, then $A$ is $a$-adically separated due to Krull intersection theorem.
\end{rem}

\begin{proof}[{\rm Proof of Theorem \ref{thm:line}}]
Recall that $\sigma_{a,x}$ is the cohomology class that comes from the section
\begin{equation}\label{eq:line-class}
  f_a=1+2w_a,\quad
  w_a^2+w_a=\frac{a}{x(x-1)},\quad w_a\equiv0\pmod\pi
\end{equation}
on $D(x(x-1))$.
Hensel's lemma gives the unique indicated root $w_a$: the derivative
of $T^2+T-a/(x(x-1))$ at $T=0$ is $1$ modulo $\pi$.
We have
\begin{equation}\label{eq:line-square}
  f_a^2=1+\frac{4a}{x(x-1)}.
\end{equation}

We first assume that $R$ is complete. Write $K=\Frac(R)$, $k=R/\pi R$, and
$F=K(x)$. For $g=x$, $x-1$, or $x(x-1)$, put
\[
  E_g=(R[x,1/g])^h_{(\pi)}.
\]
These are Noetherian and $\pi$-torsion-free, since henselization is flat and preserves Noetherianity, and moreover, $\pi$-adically separated since $\pi E_g\subseteq\Jac(E_g)$.
Moreover,
\[
  E_g/\pi E_g=k[x,1/\ovl{g}],
\]
which is an integral domain. 
Hence by Lemma \ref{lem:gluingOfDomainness} (2), each $E_g$ is an integral domain.
Moreover, by Lemma \ref{lem:gluingOfDomainness} (1), the restriction maps
\[
  E_x\longrightarrow E_{x(x-1)},\quad
  E_{x-1}\longrightarrow E_{x(x-1)}
\]
are injective, since mod $\pi$ reduction of them are injective.
All these rings may therefore be viewed inside a
common field. 
Since henselization is constructed by a filtered colimit of \'etale algebras (see, e.g., \cite[\tagref{0A02}]{stacks-project}), each of their elements is algebraic over $F$.

Suppose that $f_a$ is in the image of the two restriction maps, say
\begin{equation}\label{eq:hypothetical-split}
  f_a=u+v,\quad u\in E_x,\quad v\in E_{x-1}.
\end{equation}
We will show that $f_a$ belongs to neither $F(u)$ nor $F(v)$,
contradicting Lemma~\ref{lem:involutions}.

Hensel's lemma gives $\alpha\in\pi R$ satisfying
\[
  \alpha^2-\alpha+4a=0.
\]
Set $\beta=1-\alpha$. Then
\begin{equation}\label{eq:branch-factorization}
  f_a^2=\frac{(x-\alpha)(x-\beta)}{x(x-1)},\quad
  \alpha\beta=4a\ne0,
\end{equation}
and $\alpha-\beta=2\alpha-1$ is a unit of $R$.

There is a natural map
\begin{equation}\label{eq:formal-embedding}
  E_{x-1}\longrightarrow R\dbl x\dbr.
\end{equation}
To see this, $x-1$ is a unit in $R\dbl x\dbr$, and $R\dbl x\dbr$ is
$\pi$-adically complete, hence is $(\pi)$-adically henselian.
The universal property extends the map from $R[x,(x-1)^{-1}]$.
The mod $\pi$ reduction of \eqref{eq:formal-embedding} is the injective map
$k[x,(x-1)^{-1}]\to k\dbl x\dbr$. Injectivity of
\eqref{eq:formal-embedding} follows from Lemma \ref{lem:gluingOfDomainness} (1), since $R\dbl x\dbr$ is
$\pi$-torsion-free.

As $R$ is complete and $\alpha\in\pi R$, substitution
$x=\alpha+z$ is an isomorphism from $R\dbl x\dbr$ to $R\dbl z\dbr$.
It follows that \eqref{eq:formal-embedding} induces a field embedding
\[
  F(v)\longrightarrow K\dpl z\dpr,\quad x\longmapsto\alpha+z.
\]
Under this embedding the rational function on the right-hand side
of \eqref{eq:branch-factorization} has $z$-valuation exactly one:
the numerator has a simple zero and the denominator has constant
term $\alpha(\alpha-1)=-4a\ne0$ in $K$. A square in $K\dpl z\dpr$ has
even valuation. Consequently $f_a\notin F(v)$.

Likewise $E_x$ embeds into $R\dbl x-1\dbr$. Substituting $x=\beta+z$,
which is permitted because $\beta-1\in\pi R$, proves
$f_a\notin F(u)$. This contradicts Lemma~\ref{lem:involutions}.
Thus \eqref{eq:hypothetical-split} is impossible. The injection
\eqref{eq:cech-injection} proves the desired non-vanishing.

For a non-complete henselian DVR $R$, a hypothetical decomposition
\eqref{eq:hypothetical-split} maps, by functoriality of henselization,
to the same decomposition over $\what{R}$. The distinguished root
$w_a$ maps to the distinguished root over $\what{R}$ by uniqueness,
and $a$ remains nonzero; recall that $R\to\what{R}$ is faithfully flat, hence is injective.
This contradicts the complete case. 
\end{proof}

\begin{cor}\label{cor:line-torsion}
The class $\sigma_{a,x}$ is not annihilated by any
nonzero element of $R$. Moreover
$\HH^1(L_R,\cO)$ is a nonzero $\Frac(R)$-vector space and is not
finitely generated as an $R$-module.
\end{cor}

\begin{proof}
Apply Proposition~\ref{prop:dvr-cohomology} and
Remark~\ref{rem:not-finite}.
\end{proof}

\section{Proof of the Noetherian case}\label{sec:generalNoetherian}
Let us first prove a useful lemma, which will be used in the following discussion.

\begin{lem}[Trace and first cohomology]\label{lem:trace}
Let $(A,I)$ be a Noetherian henselian pair with $A$ a normal domain
of characteristic zero. Let $A\hookrightarrow B$ be a finite
injective homomorphism. For the induced finite morphism
\[
  \varphi:\Sph(B,IB)\longrightarrow\Sph(A,I),
\]
the natural map
\[
  \HH^1(\Sph(A,I),\cO)\longrightarrow\HH^1(\Sph(B,IB),\cO)
\]
is injective. 
\end{lem}

\begin{proof}
Put $F=\Frac(A)$ and choose a prime $\fq$ of $B$ lying over $(0)$.
Then $L=\Frac(B/\fq)$ is a finite extension of $F$. The map
\[
  T:B\longrightarrow F,\quad
  b\longmapsto\Tr_{L/F}(\ovl{b})
\]
has image in $A$. Indeed, $\ovl{b}$ is integral over $A$, so its
conjugates and their sum are integral over $A$; this sum belongs
to $F$, and $A$ is integrally closed in $F$.
The resulting $A$-linear map satisfies
\[
  T\circ\iota=n\,\mathrm{id}_A,
  \quad n=[L:F]>0,
\]
where $\iota\colon A\injto B$ is the given inclusion.

By Proposition~\ref{prop:background2} (2), sheafification gives
maps whose composition is multiplication by $n$:
\[
  \cO_{\Sph(A,I)}\longto\varphi_{\ast}\cO_{\Sph(B,IB)}
    \stackrel{\,T\,}{\longto}\cO_{\Sph(A,I)}.
\]
Since $A$ is a characteristic-zero domain, $n$ is a non-zero-divisor.
Lemma~\ref{lem:torsion} shows that its action on $\HH^1(\cO)$
is injective. Therefore the first arrow induces an injection on
first cohomology. The low-degree Leray sequence (the exact sequence $0\to E^{1,0}_2\to E^1$ in the Leray spectral sequence \cite[\tagref{01F2}]{stacks-project}) supplies a further injection
\[
  \HH^1(\Sph(A,I),\varphi_{\ast}\cO_{\Sph(B,IB)})
  \injlongto
  \HH^1(\Sph(B,IB),\cO).
\]
Their composition is the claimed map.
\end{proof}

Let $(A,I)$, $(B,IB)$, and $\fp\in V(I)$ be as in Theorem \ref{thm:main1}.
We know that $B_0=B\otimes_A\ka(\fp)$ is a finite type algebra over the field $\ka(\fp)$; let $d$ be the dimension of $B_0$, which we know to be positive.

\begin{dfn}\label{dfn:normalization}
A $d$-tuple $(b_1,\ldots,b_d)$ of elements of $B$ is called a {\em normalization tuple} if the $\ka(\fp)$-algebra homomorphism 
$$
\ka(\fp)[t_1,\ldots,t_d]\longto B_0
$$
that maps $t_i$ ($i=1,\ldots,d$) to the image of $b_i$ is injective and finite.
\end{dfn}

The existence of a normalization tuple in $B$ is immediate from classical Noether normalization theorem if $B\to B_0$ is surjective (i.e., $\fp$ is maximal in $A$). In general, one can choose $b_i\in B$ and $c_i\in A\setminus\fp$ ($i=1,\ldots,d$) such that $\ovl{b}_1/\ovl{c}_1,\ldots,\ovl{b}_d/\ovl{c}_d$ (where $\ovl{\,\cdot\,}$ indicates mod $\fp$ reduction) give coordinates in Noether normalization of $B_0$. Then, since $\ovl{c}_i\in\ka(\fp)$ are constants, $(b_1,\ldots,b_d)$ does the work.

The goal of this section is to prove the following theorem.
\begin{thm}\label{thm:noetherian_case}
In the situation as in Theorem {\rm \ref{thm:main1}}, suppose that $A$ is Noetherian. Let $(b_1,\ldots,b_d)$ be a normalization tuple of elements in $B$, and $b$ be one of the $b_i$'s. Then the cohomology class $\sigma_{a,b}$ by any non-zero $a\in I$ is non-zero in $\HH^1(X,\cO_X)$.
\end{thm}

\begin{proof}
We choose a DVR $R$ in the fractional field of $A$ and a map $A\to R$ that dominates $\fp$. Let $\pi\in R$ be a uniformizer, and $\what{R}$ be the $\pi$-adic completion of $R$. Note that, since $I\subset\fp$, the ideal $IR$ is of the form $(\pi^n)$ for some $n>0$, and hence $A\to\what{R}$ is adic. Note also that the maps $A\to A_{\fp}\to R\to \what{R}$ are all injective. In particular, any non-zero $a\in I$ has the non-zero image in $\what{R}$.

The images of $b_1,b_2,\ldots,b_d$ determine the second map of the chain of maps
\begin{equation}\label{eqn:chain}
\what{R}\{t_1\}\stackrel{\alpha}{\longto}\what{R}\{t_1,\ldots,t_d\}\stackrel{\beta}{\longto} B_{\what{R}}=(B\otimes_A\what{R})^h.
\end{equation}
By Proposition \ref{prop:background1} (2), the two maps in \eqref{eqn:chain} are hfp. Set $C=\what{R}\{t_1,\ldots,t_d\}$ and $D=B_{\what{R}}$, and let $k=\what{R}/\pi\what{R}$ be the residue field of $\what{R}$. Then the map 
\begin{equation}\label{eqn:DVR_red}
\ovl{\beta}\colon C\otimes_{\what{R}}k=k[t_1,\ldots,t_d]\longto B\otimes_Ak=D\otimes_{\what{R}}k
\end{equation}
induced from $\beta$ (i.e., $\ovl{\beta}=\beta\otimes_{\what{R}}k$) is the base change of $\ka(\fp)[t_1,\ldots,t_d]\to B_0=B\otimes_A\ka(\fp)$ by the field extension $\ka(\fp)\injto k$, and hence is injective and finite.
In particular, $\beta$ is injective (due to Lemma \ref{lem:gluingOfDomainness} (1)) and finite (due to Proposition \ref{prop:background1} (1)).
Moreover, $C=\what{R}\{t_1,\ldots,t_d\}$ is $\pi$-torsion free (i.e., $\what{R}$-flat) and $\pi$-adically separated (cf.\ Remark \ref{rem:separated}), and $C/\pi C\simeq k[t_1,\ldots,t_d]$ is an integral domain, $C$ is an integral domain (Lemma \ref{lem:gluingOfDomainness} (2)).

Since \'etale algebras over the normal ring $\what{R}[t_1,\ldots,t_d]$ are normal, and a filtered colimit of normal rings is again normal (\cite[\tagref{037D}]{stacks-project}), $C=\what{R}\{t_1,\ldots,t_d\}$ is a normal domain of characteristic zero. Hence the second map of 
\begin{equation*}\label{eqn:chain_coh}
\HH^1(\Sph \what{R}\{t_1\},\cO)\longto\HH^1(\Sph \what{R}\{t_1,\ldots,t_d\},\cO)\longto \HH^1(\Sph B_{\what{R}},\cO)
\end{equation*}
is injective by Lemma \ref{lem:trace}; the first map is also injective, since $\what{R}\{t_1\}\to\what{R}\{t_1,\ldots,t_d\}$ has the retraction map $\what{R}\{t_1,\ldots,t_d\}\to \what{R}\{t_1\}$.

Let $b$ be one of $b_i$'s, say $b=b_1$.
Now it has already been shown in Theorem \ref{thm:line} that $\sigma_{a,t_1}\in\HH^1(\Sph \what{R}\{t_1\},\cO)$ is non-zero for any non-zero $a\in I$, and in view of the diagram
$$
\xymatrix{&\HH^1(\Sph \what{R}\{t_1\},\cO)\ar@{_{(}->}[d]\\
\HH^1(\Sph(B,IB),\cO)\ar[r]&\HH^1(\Sph B_{\what{R}},\cO)}
$$
where $\sigma_{a,t_1}\in\HH^1(\Sph \what{R}\{t_1\},\cO)$ and $\sigma_{a,b}\in\HH^1(\Sph(B,IB),\cO)$ have the same image in $\HH^1(\Sph B_{\what{R}},\cO)$, we deduce that $\sigma_{a,b}\neq 0$.
\end{proof}

\section{Removal of the Noetherian hypothesis}\label{sec:general}
We now prove Theorem \ref{thm:main1} in full generality. 
Let $(A,I)$, $(B,IB)$, and $\fp\in V(I)$ be as in Theorem \ref{thm:main1}.

\begin{lem}[Flat algebraization and finite approximation]\label{na:approximation}
There exists an $A$-flat finitely presented algebra $S$ with $B\simeq S^h_{IS}$.  Moreover, for any prescribed finite subset of $A$, there are a finitely generated $\mathbf Z$-subalgebra $A_0\subseteq A$ containing that subset and a flat finitely presented $A_0$-algebra $S_0$ such that
$S\simeq S_0\otimes_{A_0}A$.
\end{lem}

\begin{proof}
The first assertion is \cite[Lemma~5.2.11]{devadas2026henselian}.
Write $A$ as the filtered union of its finitely generated $\mathbf Z$-subalgebras.  
Finite presentation descends by \cite[\tagref{05N9}]{stacks-project}, and flatness holds at a sufficiently large stage by \cite[Tag~02JO]{stacks-project}, applied with the algebra itself as the module.  
Increasing the stage includes the prescribed finite subset; flatness and finite presentation are preserved by this base change.
\end{proof}

\begin{proof}[Proof of Theorem {\rm \ref{thm:main1}}]
By considering the $A_0$'s as above for all finite subsets of $A$, taking henselization of $(A_0,I_0)$ ($I_0=I\cap A_0$) each time, and replacing $A_0$ by the image of $A^h_0$ in $A$ (which is again henselian by Lemma \ref{lem:henselianProperties} (1)), one has:
\begin{itemize}
\item a filtered direct system $\{(A_{\la},I_{\la})\}_{\la\in\La}$ with $I_{\la}\neq 0$ of Noetherian henselian pairs by integral domains $A_{\la}$ that converges to $(A,I)$; 
\item an index $\la_0\in\La$ and $A_{\la_0}$-flat finitely presented $A_{\la_0}$ algebra $S_{\la_0}$ such that $S\simeq S_{\la_0}\otimes_{A_{\la_0}}A$.
\end{itemize}

For any $\la\geq\la_0$, we set $S_{\la}=S_{\la_0}\otimes_{A_{\la_0}}A_{\la}$ and $\fp_{\la}=\fp\cap A_{\la}$; if $B_{\la}$ is the $I_{\la}$-adic henselization of $S_{\la}$, we have $(B,IB)=\varinjlim_{\la\geq\la_0}(B_{\la},I_{\la}B_{\la})$.
For each $\la\geq\la_0$, $S_{\la}\otimes_{A_{\la}}\ka(\fp_{\la})$ is a finite type algebra over $\ka(\fp_{\la})$. Since, for any $\la\geq\la_0$, 
\begin{equation}\label{eqn:simeq}
\begin{split}
S_{\la}\otimes_{A_{\la}}\ka(\fp_{\la})&\simeq \big(S_{\la_0}\otimes_{A_{\la_0}}A_{\la}\big)\otimes_{A_{\la}}\ka(\fp_{\la})\\ &\simeq S_{\la_0}\otimes_{A_{\la_0}}\ka(\fp_{\la})\simeq \big(S_{\la_0}\otimes_{A_{\la_0}}\ka(\fp_{\la_0})\big)\otimes_{\ka(\fp_{\la_0})}\ka(\fp_{\la})
\end{split}
\end{equation}
and similarly $S\otimes_A\ka(\fp)\simeq\big(S_{\la_0}\otimes_{A_{\la_0}}\ka(\fp_{\la_0})\big)\otimes_{\ka(\fp_{\la_0})}\ka(\fp)$, we have $\dim(B_{\la}\otimes_{A_{\la}}\ka(\fp_{\la}))=\dim(S_{\la}\otimes_{A_{\la}}\ka(\fp_{\la}))=d$ for all $\la\geq\la_0$.

Take a normalization tuple $(b_1,\ldots,b_d)$ in $S_{\la_0}$, and we denote its images in $S_{\la}$, $S$, $B_{\la}$, and $B$ by the same symbol. 
One deduces by \eqref{eqn:simeq} that the map $\ka(\fp_{\la})[t_1,\ldots,t_d]\to S_{\la}\otimes_{A_{\la}}\ka(\fp_{\la})$ determined by the tuple $(b_1,\ldots,b_d)$ is the base change by the field extension $\ka(\fp_{\la_0})\injto\ka(\fp_{\la})$ of the finite injective map $\ka(\fp_{\la_0})[t_1,\ldots,t_d]\injto S_{\la_0}\otimes_{A_{\la_0}}\ka(\fp_{\la_0})$, which is determined similarly by $(b_1,\ldots,b_d)$.
Hence $\ka(\fp_{\la})[t_1,\ldots,t_d]\to S_{\la}\otimes_{A_{\la}}\ka(\fp_{\la})$ is finite injective, and thus $(b_1,\ldots,b_d)$ is a normalization tuple in $S_{\la}$ and $B_{\la}$ for any $\la\geq\la_0$; in a similar vein, it is a normalization tuple also in $S$ and $B$.

Now we set $b=b_1$, and denote its image in $B_{\la}$ by $b_{\la}$.
Replacing $\la_0$ by a sufficiently larger one if necessary, one can choose $a\in I_{\la_0}$ whose image in $I$ is non-zero.
We already know that the cohomology class $\sigma_{a,b_{\la}}$ is non-zero in $\HH^1(\Sph B_{\la},\cO)$. 
Now the projective limit of $\{\Sph B_{\la}\}_{\la\geq\la_0}$ in the category of locally ringed spaces is isomorphic to $\Sph B$ (\cite[\tagref{01YW}]{stacks-project}); in particular, the underlying topological space of $X=\Sph B$ coincides with the projective limit of the underlying topological spaces of $X_{\la}=\Sph B_{\la}$ (cf.\ \cite[Chap.~0, Prop.~4.1.10]{fujiwara_kato_2018}). Since $\cO_X\simeq \varinjlim_{\la\geq \la_0}p^{-1}_{\la}\cO_{X_{\la}}$, where $p_{\la}\colon X\rightarrow X_{\la}$ is the projection, we have 
\begin{equation}\label{eqn:isom_coh}
\HH^1(\Sph B,\cO)\simeq\varinjlim_{\la\geq\la_0}\HH^1(\Sph B_{\la},\cO),
\end{equation}
(cf.\ \cite[Chap.~0,Prop.~3.1.16]{fujiwara_kato_2018}).
Since $\sigma_{a,b_{\la}}$ forms a compatible system of elements in the right-hand colimit of \eqref{eqn:isom_coh}, we deduce that $\sigma_{a,b}$ in the left-hand cohomology is non-zero in $\HH^1(\Sph B,\cO)$.
\end{proof}

\begin{proof}[Proof of Corollary {\rm \ref{cor:general}}]
Set $A'=A/P$, and $B'=B/PB$.
Then $(A',IA')$ is a henselian pair with $A'$ an integral domain of characteristic zero, $IA'\neq 0$, and $(B',IB')$ is an $A'$-flat hfp algebra. Moreover, $\fp'=\fp A'$ is prime ideal of $A'$ such that $\dim(B'\otimes_{A'}\ka(\fp'))>0$.

Then, as we have seen in the proof of Theorem \ref{thm:main1}, there exists $a'\in IA'$ and $b'\in B'$ such that $\sigma_{a',b'}$ gives a non-zero cohomology class in $\HH^1(\Sph B',\cO)$.
Take a lift $b\in B$ of $b'$ and $a\in I$ of $a'$.
Then $\sigma_{a,b}$ is mapped to $\sigma_{a',b'}$ by
$$
\HH^1(\Sph B,\cO)\longto\HH^1(\Sph B',\cO).
$$
Hence we have $\sigma_{a,b}\neq 0$, which proves the corollary.
\end{proof}

\section{Examples}\label{sec:exa}

\begin{exa}[The flatness hypothesis]\label{exa:nonflat}
The flatness hypothesis in Theorem~\ref{thm:main1}
cannot be omitted.
Let $R$ be a complete DVR of characteristic zero, with
uniformizer $\pi$ and residue field $k$, and set
\[
  A=R,\quad I=\pi R,\quad B=k[x].
\]
The $R$-algebra structure on $B$ is induced by $R\to k$.
Since $B\simeq R\{x\}/\pi R\{x\}$, the algebra $B$ is of henselian finite presentation over $R$.
Moreover, $IB=0$, so $(B,IB)$ is a henselian pair.
At the closed point $\fp=(\pi)$, we have $B\otimes_R\ka(\fp)\simeq k[x]$, hence 
$$  
\dim(B\otimes_R\ka(\fp))=1.
$$
However, $B$ is not $R$-flat, since $\pi B=0$ and $B\ne0$.
Finally,
\[
  X=\Sph(B,IB)=\Spec k[x]
\]
as locally ringed spaces, and hence
\[
  \HH^q(X,\cO_X)=0\quad(q>0).
\]
\end{exa}

\begin{exa}[Zero-dimensional fibers]\label{exa:zero-dimensional}
The existence of a positive-dimensional fiber in Theorem~\ref{thm:main1} cannot be omitted either.
Let $R$ be a complete DVR of characteristic zero, with
uniformizer $\pi$, and take
\[
  A=B=R,\quad I=\pi R,
\]
with the identity map $A\to B$.
Then $B$ is $A$-flat and of henselian finite presentation.
For every $\fp\in\Spec A$, we have
\[
  B\otimes_A\ka(\fp)\simeq\ka(\fp),
\]
so every fiber is nonempty and zero-dimensional.
On the other hand, the underlying topological space of
$X=\Sph(R,\pi R)$ consists of a single point.
The global sections functor on abelian sheaves on $X$
is therefore exact, and consequently
\[
  \HH^q(X,\cO_X)=0\quad(q>0).
\]
\end{exa}

\begin{exa}[The finite presentation hypothesis]
\label{exa:power-series-no-finiteness}
The henselian finite presentation hypothesis in
Theorem~\ref{thm:main1} cannot be omitted, even when the total
ring is a regular Noetherian complete local domain.
Let $R$ be a complete DVR of characteristic zero, with uniformizer
$\pi$ and residue field $k$, and set
\[
  A=R,\quad I=\pi R,\quad B=R\dbl x\dbr.
\]
Then $B$ is a regular Noetherian complete local domain.
It is $\pi$-torsion free, hence flat, over the DVR $R$.
Moreover, $B$ is $\pi$-adically complete, since
\[
  B/\pi^nB\simeq(R/\pi^nR)\dbl x\dbr,
  \quad
  B\xrightarrow{\sim}\varprojlim_{n\geq1}B/\pi^nB,
\]
where the second isomorphism follows coefficientwise from the
completeness of $R$.
Thus $(B,\pi B)$ is a henselian pair.
At the closed point $\fp=(\pi)$ of $\Spec R$, the fiber is
\[
  B\otimes_R\ka(\fp)\simeq k\dbl x\dbr,
  \quad
  \dim(B\otimes_R\ka(\fp))=1.
\]

Nevertheless, for $X=\Sph(B,\pi B)$, one has
\[
  \HH^q(X,\cO_X)=0\quad(q>0).
\]
Indeed, the underlying topological space of $X$ is
$\Spec k\dbl x\dbr$.
If $s$ is its closed point, the only open neighborhood of $s$
is $X$ itself. Consequently, for every sheaf $\cF$ of abelian groups on $X$,
\[
  \Gamma(X,\cF)\simeq\cF_s.
\]
Since taking a stalk is exact, the global sections functor is
exact. In fact, all positive-degree cohomology groups of every
abelian sheaf on $X$ vanish.
\end{exa}

\begin{rem}\label{rem:not_ft}
Note that, in the situation as above, $B$ is not even of henselian finite type over $R$.
Otherwise, its reduction $B/\pi B=k\dbl x\dbr$ would be a finite type $k$-algebra, which is absurd (for, e.g., $k\dbl x\dbr$ is not Jacobson).
\end{rem}

\begin{exa}[A reduced non-flat example with injective structure map]
\label{exa:reduced-nonflat-pi-x}
The flatness hypothesis in Theorem~\ref{thm:main1} cannot be
replaced by injectivity of the structure map, even when the
total ring is reduced and the extended ideal of definition is
not nilpotent.
Let $R$ be a complete DVR of characteristic zero, with uniformizer
$\pi$ and residue field $k$, and set
\[
  A=R,\quad I=\pi R,\quad B=R[x]/(\pi x).
\]
We use $x$ also for its image in $B$.
There is a natural isomorphism
\[
  B\xrightarrow{\sim}R\times_k k[x],
  \quad
  f(x)\longmapsto\bigl(f(0),\overline{f}(x)\bigr),
\]
where the fiber product is taken with respect to the residue map
$R\to k$ and evaluation at $x=0$ on $k[x]$.
Under this isomorphism, $\pi^nB$ corresponds to
$\pi^nR\times\{0\}$ for every $n\geq1$.
Hence
\[
  B/\pi^nB\simeq(R/\pi^nR)\times_k k[x],
  \quad
  B\xrightarrow{\sim}\varprojlim_{n\geq1}B/\pi^nB.
\]
Thus $B$ is already $\pi$-adically complete, and hence $(B,\pi B)$ is henselian.
Since henselization commutes with passage to a quotient (Lemma \ref{lem:henselianProperties} (1)), we obtain
\[
  B\simeq R\{x\}/(\pi x).
\]
In particular, $B$ is of henselian finite presentation over $R$;
it is also finitely presented as an ordinary $R$-algebra.

Evaluation at $x=0$ gives a retraction $B\to R$, so $R\to B$ is injective. Since $B\simeq R\times_k k[x]$ embeds into the product of reduced rings $R\times k[x]$, $B$ is reduced.
Also $\pi^n\neq0$ in $B$ for every $n\geq1$, so $\pi B$ is
not nilpotent.
On the other hand, $x\neq0$ and $\pi x=0$ in $B$, whence $B$
is not $R$-flat.
The closed fiber is $B/\pi B\simeq k[x]$, and is therefore one-dimensional.

We show that $X=\Sph(B,\pi B)$ nevertheless satisfies
\[
  \HH^q(X,\cO_X)=0\quad(q>0).
\]
Identify $|X|$ with $\mathbb A^1_k$, and let
$i:\{0\}\hookrightarrow |X|$ be the inclusion of the origin.
Give $R$ its $B$-module structure through evaluation at $x=0$.
There is then an exact sequence of $B$-modules
\[
  0\longrightarrow R
   \xrightarrow{\ r\mapsto\pi r\ }B
   \longrightarrow k[x]\longrightarrow0.
\]
The first arrow is $B$-linear because $\pi x=0$.
By Proposition~\ref{prop:background2}~(1), this induces an exact
sequence of $\cO_X$-modules.
To identify its outer terms, apply
Proposition~\ref{prop:background2}~(2) to the finite quotient maps
$B\to R=B/(x)$ and $B\to k[x]=B/\pi B$.
The henselian spectrum $\Sph(R,\pi R)$ has a single point and
structure ring $R$, whereas
$\Sph(k[x],0)=\Spec k[x]$ as locally ringed spaces.
It follows that the resulting sequence is
\[
  0\longrightarrow i_*R
   \longrightarrow\cO_X
   \longrightarrow\cO_{\mathbb A^1_k}
   \longrightarrow0,
\]
where $i_*R$ is the skyscraper sheaf with value $R$ at the origin,
and $\cO_{\mathbb A^1_k}$ is viewed as a sheaf on the same
underlying topological space $|X|$.
Since the sheaf $i_*R$ is flasque, and the higher cohomology of
$\cO_{\mathbb A^1_k}$ vanishes, the cohomology long exact sequence gives the claimed vanishing.
\end{exa}

\section{Non-vanishing for henselian affinoids}\label{sec:affinoids}
In this section, we will prove Theorem \ref{thm:general-henselian-affinoid-nonvanishing}, which asserts the non-vanishing of the structure-sheaf cohomology for {\em henselian affinoids} by applying the non-vanishing results obtained so far on Henselian schemes. For basic information on Henselian rigid geometry, see \cite{kato2017henselianrigid}.

Let $V$ be a valuation ring of height one, and let
$a\in\fm_V\setminus\{0\}$ be such that $(V,aV)$ is henselian.
Write $K=\operatorname{Frac}(V)=V[1/a]$ and $k=V/\fm_V$.
Henselian affinoid algebras and their associated rigid spaces
are understood in the sense of
\cite[\S\S 3.1--3.2]{kato2017henselianrigid}.
For a henselian rigid space $\mathcal X$, cohomology means
cohomology on its Zariski--Riemann space
$\langle\mathcal X\rangle$ with the \emph{rigid} structure sheaf
$\cO_{\mathcal X}$, rather than the integral structure sheaf.

\begin{prop}[Persistence under passage to the associated rigid spaces]
\label{prop:rigidification-h1-injective}
Let $B$ be a $V$-flat henselian finite type $V$-algebra, and set
\[
  X=\Sph(B,aB),\quad \mathcal X=X^{\mathrm{rig}}.
\]
For the specialization map $s:\langle\mathcal X\rangle\to X$,
there is a canonical isomorphism
\begin{equation}\label{eq:rigidification-direct-image}
  s_*\cO_{\mathcal X}\simeq\cO_X[1/a].
\end{equation}
Consequently, the natural map
\begin{equation}\label{eq:rigidification-h1-injective}
  \HH^1(X,\cO_X)\longrightarrow
  \HH^1(\langle\mathcal X\rangle,\cO_{\mathcal X})
\end{equation}
is injective. Both groups are naturally $B[1/a]$-modules,
and this map is $B[1/a]$-linear.
\end{prop}

\begin{proof}
Put $Z=\langle\mathcal X\rangle$. By
\cite[\S 3.2]{kato2017henselianrigid}, $Z$ is the inverse limit
of the admissible blow-ups $h:Y\to X$, and its integral structure
sheaf is
\[
  \cO^{\mathrm{int}}_{\mathcal X}
  \simeq\varinjlim_{Y\to X}p_Y^{-1}\cO_Y,
\]
where $p_Y:Z\to Y$ denotes the projection.
The blow-up charts of a $V$-flat model are $a$-torsion free,
and henselization preserves flatness. By
\cite[\tagref{0539}]{stacks-project}, these models are
$V$-flat, and $a$ is a non-zero-divisor in
$\cO^{\mathrm{int}}_{\mathcal X}$. The definition of the rigid
structure sheaf therefore gives
\[
  \cO_{\mathcal X}=\cO^{\mathrm{int}}_{\mathcal X}[1/a].
\]

Let $U=D(g)\subset X$ be a distinguished open, and write
$C=(B_g)^h_{aB_g}$, so that $U=\Sph(C,aC)$.
The spaces $Y$ are spectral and the transition maps are
spectral: their underlying spaces are those of their special
fibres over $k$, which are quasi-compact and quasi-separated
schemes of finite type over $k$.
The degree-zero case of
\cite[\tagref{0A37}]{stacks-project}, together with
\cite[\tagref{01FF}]{stacks-project} for localization, yields
\begin{equation}\label{eq:rigidification-local-sections}
  \Gamma(s^{-1}U,\cO_{\mathcal X})
  \simeq\varinjlim_{h\colon Y\to X}
       \Gamma(h^{-1}U,\cO_Y)[1/a]
  \simeq C[1/a].
\end{equation}
The last isomorphism follows from
\cite[Proposition~3.3]{kato2017henselianrigid}, applied to
$h^{-1}U\to U$, which is an admissible blow-up.
All these isomorphisms are canonical and respect restriction.
On the other hand,
$\Gamma(U,\cO_X[1/a])=C[1/a]$ by
\cite[\tagref{01FF}]{stacks-project} in degree zero.
This proves \eqref{eq:rigidification-direct-image}.

Multiplication by $a$ is an isomorphism on $\HH^1(X,\cO_X)$
by Proposition~\ref{prop:dvr-cohomology}. This follows from
the exact sequence
\[
  0\longrightarrow\cO_X\xrightarrow{\ a\ }\cO_X
   \longrightarrow\cO_{\Spec(B/aB)}\longrightarrow0
\]
and affine scheme cohomology. Since $X$ is spectral,
\cite[\tagref{01FF}]{stacks-project} also gives
\[
  \HH^1(X,\cO_X[1/a])
  \simeq\HH^1(X,\cO_X)[1/a]
  \simeq\HH^1(X,\cO_X).
\]
Finally, the low-degree Leray sequence for the morphism of
ringed spaces $s$, using the rigid structure sheaf on its source,
gives an injection
\[
  \HH^1(X,s_*\cO_{\mathcal X})
  \lhook\joinrel\longrightarrow
  \HH^1(Z,\cO_{\mathcal X});
\]
see \cite[\tagref{01F2}]{stacks-project}.
Combining the preceding identifications proves
\eqref{eq:rigidification-h1-injective}, with its asserted linearity.
\end{proof}

\begin{proof}[Proof of Thorem {\rm \ref{thm:general-henselian-affinoid-nonvanishing}}]
By \cite[\S 3.1]{kato2017henselianrigid}, $\mathcal A$ admits
a $V$-flat model $B$ of henselian finite presentation with
\[
  \mathcal A=B[1/a],\quad
  \mathcal X=\bigl(\Sph(B,aB)\bigr)^{\mathrm{rig}}.
\]
More explicitly, one lifts a finite set of generators of the
defining ideal of $\mathcal A$ to $V\{T_1,\ldots,T_n\}$
and replaces the lifted ideal by its $a$-saturation, which
is still finitely generated.

We claim that $\dim(B/\fm_VB)>0$.
Since $V$ has height one,
\[
  \sqrt{aV}=\fm_V,
  \quad
  \sqrt{aV[T_1,\ldots,T_n]}=\fm_VV[T_1,\ldots,T_n].
\]
Henselization depends only on the radical of the ideal
\cite[\tagref{0F0L}]{stacks-project}.
Thus $(V,\fm_V)$ and $(B,\fm_VB)$ are henselian pairs,
and $B$ is hfp over $(V,\fm_V)$ as well.
The algebra $B/\fm_VB$ is a finite type $k$-algebra, and
it is nonzero because $B\ne0$ and
$\fm_VB\subseteq\operatorname{Jac}(B)$.
If its dimension were zero, it would be finite over $k$.
Proposition~\ref{prop:background1}\,\textup{(1)}, applied to
$(V,\fm_V)$, would then imply that $B$ is finite over $V$.
Hence $\mathcal A=B[1/a]$ would be finite over $K$, contrary
to $\dim\mathcal A>0$. This proves the claim.

Theorem~\ref{thm:main1}, applied to $(V,aV)$, $B$, and
$\mathfrak p=\fm_V$, now gives
\[
  \HH^1(\Sph(B,aB),\cO)\ne0.
\]
The result follows from
Proposition~\ref{prop:rigidification-h1-injective}.
\end{proof}

\begin{rem}[Dimension zero]
\label{rem:zero-dimensional-henselian-affinoids}
Let $\mathcal A\ne0$ be a zero-dimensional
henselian affinoid algebra. Henselian Noether normalization
\cite[\S 3.1]{kato2017henselianrigid} implies that $\mathcal A$
is finite over $K$. Its reduction is a finite product of
finite field extensions of $K$. Since $V$ is henselian,
its valuation extends uniquely to each such extension.
The valuation-theoretic description of the Zariski--Riemann
space in \cite[\S 3.2]{kato2017henselianrigid} consequently
shows that $\langle\mathcal X\rangle$ is finite and discrete;
nilpotents do not affect this space. Hence
\[
  \HH^q(\langle\mathcal X\rangle,\cO_{\mathcal X})=0
  \quad(q>0).
\]
Thus, in characteristic zero, for a nonempty finite type
henselian affinoid one has
\[
  \HH^1(\langle\mathcal X\rangle,\cO_{\mathcal X})\ne0
  \quad\Longleftrightarrow\quad \dim\mathcal A>0.
\]
\end{rem}

\section*{AI disclosure}
Our project began with an attempt to understand \cite{greco1981quasi}. We used ChatGPT and Fable 5 to identify potentially problematic points in the arguments of \cite{greco1981quasi} and to search for possible counterexamples. This process led us to a very simple counterexample. Motivated by this example, we were led to suspect that, more generally, degree-one cohomology should remain nonzero under suitable hypotheses. We then used ChatGPT 6 Astra to develop this idea into the results presented in this paper.

Thus, this paper was written with substantial assistance from AI tools. At the same time, the authors carefully checked, reorganized, and simplified the AI-assisted arguments at various stages. In particular, the formula corresponding to \eqref{eq:intro-class}, which plays an important role in the argument, was first identified by ChatGPT 6 Astra.

Detailed records of our conversations with the AI tools used in this work are publicly available at \url{https://github.com/utaka233/affine_henselian_schemes_20260925/}.

\section*{Acknowledgments}
This project was carried out as part of the research activities at the ZEN Mathematics Center (ZMC), the Institute for Pure Mathematics at ZEN University. The first author was partially supported by grants JSPS KAKENHI \#22H01112, 23H01064, 25H00587. The authors would like to thank Prof. Hajime Kaji for helpful discussions on the use of AI.

\frenchspacing
\bibliographystyle{plain}
\bibliography{henselian}
\end{document}